\documentclass[12pt]{article} 

\usepackage[margin=3cm]{geometry}
\usepackage[english]{babel}
\usepackage{hyperref}
\usepackage{amsmath}
\usepackage{amsthm}
\usepackage{amssymb}
\usepackage{enumerate}
\usepackage[utf8]{inputenc}
\usepackage[T1]{fontenc}
 \usepackage[affil-it]{authblk}
 
\usepackage{tikz}
\usetikzlibrary{arrows}
\usetikzlibrary{decorations.markings}
\pgfdeclarelayer{nodelayer}
\pgfdeclarelayer{edgelayer}
\pgfsetlayers{nodelayer,main,edgelayer}
\tikzstyle{none}=[draw=none]   
\tikzstyle{bigtiparrow}=[->,thick, >=angle 90]
\tikzstyle{bigtiparrow2}=[->,thick, >=angle 90,preaction={draw=white, -,line width=6pt}]
\tikzstyle{lrarrow}=[<->,thick, >=angle 90,preaction={draw=white, -,line width=6pt}]

\tikzstyle{new}=[rectangle,fill=white,draw=white, inner sep=2pt]
\tikzstyle{new2}=[rectangle,fill=white,draw=white, inner sep=6pt]

\DeclareMathOperator{\rng}{\mathrm{rng}}
\DeclareMathOperator{\supp}{\mathrm{supp}}

\newcommand{\Aut}{\mathrm{Aut}}

  \newcommand{\N}{\mathbb N}
  
  \newcommand{\Z}{\mathbb Z}
  
  \newcommand{\LL}{\mathrm L}

 \newcommand{\dom}{\mathrm{dom}\;}

  \newcommand{\eps}{\epsilon}

  \newcommand{\inv}{^{-1}}

  \renewcommand{\geq}{\geqslant}

  \newcommand{\la}{\left\langle}
  \newcommand{\ra}{\right\rangle}

\newtheorem{thm}{Theorem}
\newtheorem{cor}[thm]{Corollary}
\newtheorem{lem}[thm]{Lemma}
\newtheorem{prop}[thm]{Proposition}

\theoremstyle{definition}

\newtheorem{df}[thm]{Definition}
\newtheorem*{rmq}{Remark}

\newtheorem{thmi}{Theorem}

\renewcommand*{\thefootnote}{\fnsymbol{footnote}}

\title{Connected Polish groups with ample generics}
\author{Adriane Kaïchouh and François Le Maître\footnote{Research supported by the Interuniversity Attraction Pole DYGEST and Projet ANR-14-CE25-0004 GAMME.}}

\begin{document}

\setcounter{footnote}{1}

\maketitle
\renewcommand*{\thefootnote}{\arabic{footnote}}

\begin{abstract}
In this paper, we give the first examples of connected Polish groups that have ample generics, answering a question of Kechris and Rosendal. We show that any Polish group with ample generics embeds into a connected Polish group with ample generics and that full groups of type III hyperfinite ergodic equivalence relations have ample generics. We also sketch a proof of the following result: the full group of any type III ergodic equivalence relation has topological rank 2.
\end{abstract}

Ample generics were introduced by Kechris and Rosendal \cite{MR2308230} as a fundamental property underlying many phenomenons observed for various non-archimedean\footnote{A topological group is non-archimedean if it admits a basis of neighborhoods of the identity consisting of open subgroups. Every non-archimedean Polish group is a closed subgroup of the Polish group $\mathfrak S_\infty$ of all permutations of the integers (see \cite[Thm. 1.5.1]{MR1425877}). The group $\mathfrak S_\infty$ is endowed with the topology of pointwise convergence on the discrete set $\N$.} Polish groups. A Polish group has \textbf{ample generics} if for every $n\in\N$, the diagonal conjugacy action of $G$ on $G^n$ has a comeager\footnote{A subset of a Polish space is comeager if it contains a dense countable intersection of open sets, so that it is large in the Baire category sense.} orbit. In particular, $G$ must have a comeager conjugacy class, which excludes locally compact groups from this class of groups as was recently shown by Wesolek \cite{Wesolek:2013it}. 

Yet, as far as non-archimedean Polish groups are concerned, a number of interesting examples arise: the group $\mathfrak S_\infty$ of all permutations of the integers, the automorphism group of the rooted $\infty$-regular tree and the homeomorphism group of the Cantor space, to name a few (for $\mathfrak S_\infty$ and $\Aut(\N^{<\N})$, see \cite{MR2308230}, for $\mathrm{Homeo}(2^\N)$, this is a result of Kwiatkowska, see \cite{MR3007646}). 

One of the main motivations for finding groups with ample generics is that they satisfy a very strong property called the automatic continuity property: if $G$ is a Polish group with ample generics, every group homomorphism from $G$ to a separable group $H$ has to be continuous (\cite[Thm. 1.10]{MR2308230}). 
The automatic continuity property means we can recover the topology on the group from its algebraic structure only. In particular, the topology on a Polish group $G$ with ample generics is the unique possible Polish group topology on $G$.

Another motivation is the Bergman property: a group $G$ has the Bergman property if every isometric $G$-action on a metric space has a bounded orbit (the action need not be continuous). This implies that every isometric affine $G$-action on a Hilbert space has a fixed point, that is, $G$ has property (FH) as a discrete group. Besides, every $G$-action on a tree fixes either a vertex or an edge ($G$ has property (FA)). In the context of non-archimedean Polish groups, Kechris and Rosendal showed that every oligomorphic group with ample generics has the Bergman property, which applies to the three examples above.

Outside the non-archimedean world, however, ample generics often fail badly. For instance, in the group of measure-preserving bijections of a standard probability space, in the unitary group of a separable Hilbert space or in the isometry group of the Urysohn space, all conjugacy classes are meager (see \cite[Sec. 4]{MR2534176}). This led Kechris and Rosendal to ask the following question: is there a Polish group with ample generics that is not a non-archimedean group? 

In this paper, we exhibit two classes of examples of connected Polish groups with ample generics. Since non-archimedean groups are totally disconnected, these examples provide a positive answer to Kechris and Rosendal's question. Note that this question has simultaneously been answered by Malicki \cite{Malicki:2015fj}. His examples differ significantly from ours. Indeed, they arise as Polishable subgroups of $\mathfrak S_\infty$ so they are totally disconnected. Here is our first result.

\begin{thmi}[Cor. \!\ref{cor:embedinample2}]\label{thm:embedinample} Every Polish group with ample generics embeds into a contractible Polish group with ample generics.
\end{thmi}

Let us briefly describe the construction behind this theorem. Given a Polish group $G$ and a standard probability space $(X,\mu)$, let $\LL^0(X,\mu,G)$ denote the group of measurable maps from $X$ to $G$. It is a contractible Polish group which embeds $G$ via constant maps. Theorem \ref{thm:embedinample} is then a consequence of the following observation.

\begin{thmi}[Cor. \ref{cor:embedinample}]\label{thm: ample generics carry to L0}
Whenever $G$ has ample generics, so does $\LL^0(X,\mu,G)$.
\end{thmi}

This theorem generalizes the well-known fact that ample generics carry to countable powers, and the proof is essentially the same.

The second class of examples we consider comes from ergodic theory, and will yield a continuum of pairwise non-isomorphic simple contractible Polish groups with ample generics. 

A Borel bijection of the standard probability space $(X,\mu)$ is called a \textbf{non-singular automorphism} of $(X,\mu)$ if for every Borel subset $A$ of $X$, we have $\mu(A)=0$  if and only if $\mu(T(A))=0$. It is \textbf{ergodic} if every $T$-invariant Borel set has measure $0$ or $1$. Given a non-singular automorphism $T$, define its \textbf{full group} to be the group $[T]$ of all non-singular automorphisms $S$ of $(X,\mu)$ which preserve every $T$-orbit, that is, for every $x\in X$, there exists $n\in\Z$ such that $S(x)=T^n(x)$. Then $[T]$ is a Polish group for the uniform metric $d_u$ defined by $d_u(S,S'):=\mu(\{x\in X:S(x)\neq S'(x)\})$.

A non-singular ergodic automorphism of $(X,\mu)$ is called \textbf{type III} if it preserves no $\sigma$-finite measure equivalent to $\mu$.

\begin{thmi}[Thm. \ref{thm:fullgrphyphasag}]\label{thm:fullgroupag1}
Let $T$ be a type III ergodic automorphism. Then the full group $[T]$ has ample generics. 
\end{thmi}

Full groups of type III ergodic automorphisms are contractible, as was shown by Danilenko in \cite{MR1367676}. Besides, Eigen (\cite{MR654590}) has proved that they are simple. Therefore, the preceding theorem provides examples of simple contractible Polish groups with ample generics, as opposed to the highly non-simple groups we obtain with Theorem \ref{thm: ample generics carry to L0}. Also note that, by a result of Miller (\cite[Thm. 8.1]{miller2004full}), these full groups satisfy the Bergman property.
 
Furthermore, Dye's reconstruction theorem is true in the type III setting: every abstract isomorphism between full groups of type III automorphisms\footnote{More generally, of type III equivalence relations (see Section  \ref{sec:typeIII} for a definition).} comes from an orbit equivalence\footnote{An \textbf{orbit equivalence} between $T$ and $T'$ is a non-singular automorphism of $(X,\mu)$ which maps $T$-orbits to $T'$-orbits bijectively. It is easily checked that $S$ is an orbit equivalence between $T$ and $T'$ if and only if $S[T]S\inv=[T']$.} (for a general theorem from which this statement follows, see 
\cite[384D]{MR2459668}). In particular, using Krieger's classification of type III ergodic automorphisms up to orbit equivalence into types III$_0$, III$_\lambda$ ($0<\lambda<1$) and III$_1$ \cite{MR0240279}, we see that Theorem  \ref{thm:fullgrphyphasag} provides a continuum of pairwise non-isomorphic simple contractible Polish groups with ample generics. 

\section{The space of measurable maps \texorpdfstring{$\LL^0(X,\mu,Y)$}{L0(X,mu,Y)}}

In this section, we recall the definition of the Polish space $\LL^0(X,\mu,Y)$ and establish a few basic lemmas on its topology.

By definition, a probability space $(X,\mu)$ is \textbf{standard} if it is isomorphic to the interval $[0,1]$ endowed with the Lebesgue measure. Throughout the paper, $(X,\mu)$ will denote a standard probability space.

\begin{df}Let $Y$ be a Polish space. Then $\LL^0(X,\mu,Y)$ is the set of Lebesgue-measurable maps from $X$ to $Y$ up to equality $\mu$-almost everywhere. 
\end{df}

We endow $\LL^0(X,\mu,Y)$ with the \textbf{topology of convergence in measure}, which says that two maps are close in this topology if they are uniformly close on a set of large measure. To be more precise, fix a compatible metric $d_Y$ on $Y$. Then for $\epsilon>0$ and $f\in\LL^0(X,\mu,G)$, let 
$$V_{\epsilon}(f):=\{g\in \LL^0(X,\mu,Y): \mu(\{x\in X: d(f(x),g(x))<\epsilon\})>1-\epsilon\}.$$
The topology generated by the family $\left(V_{\epsilon}(f)\right)_{\epsilon,f}$  is the topology of convergence in measure.

This topology does not depend on the choice of a compatible metric $d_Y$ on $Y$ (\cite[Cor. of Prop. 6]{MR0414775}). Moreover, it is Polish (\cite[Prop. 7]{MR0414775}) and contractible: assume $X=[0,1]$ and let $y_0\in Y$ be an arbitrary point. Then an explicit contraction path is given by $$\begin{array}{rll}\LL^0(X,\mu,Y)\times[0,1] & \to & \LL^0(X,\mu,Y)  \\(f,t) & \mapsto & f_t:x\mapsto\left\{\begin{array}{cl}y_0 & \text{if }x>t, \\f(x) & \text{otherwise.}\end{array}\right.\end{array}$$

The following lemma is an easy consequence of the definition of the topology of convergence in measure.

\begin{lem}\label{lem:mugreatepsopen}
Let $Y$ be a Polish space, and let $U$ be an open subset of $Y$. Then for every $\epsilon>0$, the set 
$$V_{U,\epsilon}:=\{f\in \LL^0(X,\mu,Y): \mu(\{x\in X: f(x)\in U\})>1-\epsilon\}$$
is open in $\LL^0(X,\mu,Y)$.
\end{lem}
\begin{proof}
Let $f\in V_{U,\epsilon}$ and let $A=\{x\in X: f(x)\in U\}$. Fix a compatible metric $d_Y$ on $Y$. Since $U$ is open, the set $A$ can be written as the increasing union of the sets $A_n$'s, where $A_n=\{x\in A: d_Y(f(x), Y\setminus U)>\frac 1n\}$. By assumption, the set $A$ has measure greater than $1 - \eps$, so we can find $N\in\N$ such that $\mu(A_N)>1-\epsilon$. Now, if $\delta$ is a positive real such that $\delta<\frac 1N$ and $\delta<\mu(A_N)-(1-\epsilon)$, we see that $V_\delta(f)\subseteq V_{U,\epsilon}$, hence $V_{U,\epsilon}$ is open.
\end{proof}

Given a subset $B$ of $Y$, let 
$$\LL^0(X,\mu;B):=\{f\in \LL^0(X,\mu,Y): \forall x\in X, f(x)\in B\}.$$
A subset of a topological space is called $G_\delta$ if it can be written as a countable intersection of open sets. 

\begin{lem}\label{lem:gdeltaallinG}
Let $Y$ be a Polish space, and let $B$ be a $G_\delta$ subset of $Y$. Then the set
$\LL^0(X,\mu;B)$
is a $G_\delta$ in $\LL^0(X,\mu,Y)$. \end{lem}
\begin{proof}
Write $B=\bigcap_{n\in\N} U_n$ where each $U_n$ is open. Then clearly $\LL^0(X,\mu;B) =\bigcap_{n\in\N} \LL^0(X,\mu;U_n)$. Now, $\LL^0(X,\mu;U_n)=\bigcap_{k\in\N} V_{U_n,2^{-k}}$, so it is $G_\delta$ by the previous lemma, so $\LL^0(X,\mu;B$ itself is $G_\delta$. 
\end{proof}

\begin{lem}\label{lem:densityL0}
Let $Y$ be a Polish space, and let $B$ be dense subset of $Y$, then $\LL^0(X,\mu;B)$ is dense in $\LL^0(X,\mu;Y)$.
\end{lem}
\begin{proof}
Fix a compatible metric $d_Y$ on $Y$. Since $Y$ is separable, we can find a countable subset of $B$ which is still dense in $Y$: in other words, we may as well assume that $B$ is countable. Enumerate $B$ as $\{y_n\}_{n\in\N}$, and fix $\epsilon>0$ as well as a function $f\in \LL^0(X,\mu, Y)$. For every $x\in X$, let $n(x)$ be the smallest integer $n\in\N$ such that $d_Y(f(x),y_n)<\epsilon$. It is easily checked that the map $x\mapsto n(x)$ is measurable, so that the function $g:x\mapsto y_{n(x)}$ belongs to $\LL^0(X,\mu; B)$. But by construction, we actually have $d_Y(f(x),g(x))<\epsilon$ for \textit{all} $x\in X$, and in particular $g\in V_\epsilon(f)$, which completes the proof.
\end{proof}

Recall that a subset $B$ of a Polish space $Y$ is comeager if it contains a dense $G_\delta$ set. What follows is an immediate consequence of the previous two lemmas.

\begin{lem}\label{lem:comeagerL^0}Let $Y$ be a Polish space and let $B$ be a comeager subset of $Y$. Then $\LL^0(X,\mu;B)$ is a comeager subset of $\LL^0(X,\mu;Y)$.\qed
\end{lem}

\section{\texorpdfstring{$\LL^0(X,\mu,G)$}{L0(X,mu,G)} has ample generics whenever \texorpdfstring{$G$}{G} does}

\begin{thm}\label{thm:les orbites comaigres passent aux L^0}
Let $G$ be a Polish group acting continuously on a Polish space $Y$. 
If the action of $G$ on $Y$ has a comeager orbit, then so does the action of $\LL^0(X,\mu; G)$ on $\LL^0(X,\mu;Y)$. 
\end{thm}

\begin{proof}
Let $y_0$ be an element of $Y$ whose orbit is comeager; let $\overline{y_0}$ be the corresponding constant function in $\LL^0(X,\mu; Y)$. We show that the orbit of $\overline{y_0}$ in $\LL^0(X,\mu; Y)$ is comeager. 

First, let us remark that the orbit of $\overline{y_0}$ is thus described:
\begin{equation}\label{eqn:orbit}
\LL^0(X,\mu; G) \cdot \overline{y_0} = \{f \in \LL^0(X,\mu;Y) : f(x) \in G \cdot y_0 \text{ for almost all $x\in X$}\}.
\end{equation}
Indeed, if $f\in\LL^0(X,\mu,Y)$ is in the orbit of $\overline{y_0}$, then $f$  clearly belongs to the above set. Conversely, assume that $f(x)\in G \cdot y_0$ for almost all $x\in X$. Since every Lebesgue-measurable map coincides with a Borel map on a full measure set, we may as well assume that $f$ is a Borel map. We can then change $f$ on a Borel set of measure $0$ so that $f(x)\in G \cdot y_0$ for every $x\in X$. For all $x\in X$, there exists an element $g_x$ in $G$ such that $f(x) =g_x \cdot y_0$. We would like to find those $g_x$'s in a measurable way. To this end, we apply the Jankov-von Neumann uniformization theorem (see \cite[Cor. 5.5.8]{MR1619545}) to the following Borel set
$$S = \{(x, g) \in X \times G :  f(x) = g \cdot y_0 \},$$
which projects to the whole space $X$. We thus obtain a Lebesgue-measurable map $\varphi\in\LL^0(X,\mu; G)$ whose graph is contained in $S$, that is, $f = \varphi \cdot \overline{y_0}$, hence $f$ belongs to the $\LL^0(X,\mu,G)$-orbit of $\overline{y_0}$. 

Now, using the notations introduced at the end of the previous section, we may rewrite Equation (\ref{eqn:orbit}) as
$$\LL^0(X,\mu; G) \cdot \overline{y_0}=\LL^0(X,\mu;G\cdot y_0).$$
Since $G\cdot y_0$ is comeager, so is the set $\LL^0(X,\mu;G\cdot y_0)$ by Lemma \ref{lem:comeagerL^0}. Thus, the orbit of $\overline y_0$ is comeager in $\LL^0(X,\mu;Y)$.
\end{proof}

Recall that a topological group $G$ has \textbf{ample generics} if for all $n\in\N$, the diagonal conjugacy action of $G$ on $G^n$ has a comeager orbit, where the diagonal conjugacy action is given by $g\cdot (g_1,...,g_n)=(gg_1g\inv,...,gg_ng\inv)$.

\begin{cor}\label{cor:embedinample}
Let $G$ be a Polish group with ample generics. Then the group $\LL^0(X,\mu; G)$ also has ample generics. 
\end{cor}

\begin{proof}
Fix $n\in\N$. There is a natural identification between the groups $\LL^0(X,\mu;G)^n$ and $\LL^0(X,\mu;G^n)$, so it suffices to apply Theorem \ref{thm:les orbites comaigres passent aux L^0} to the diagonal action of $G$ on $G^n$.
\end{proof}

Since $G$ embeds into the contractible group $\LL^0(X,\mu; G)$ via constant maps, we also have the following result.

\begin{cor}\label{cor:embedinample2}
Let $G$ be a Polish group with ample generics. Then $G$ embeds into a contractible Polish group with ample generics.  \qed
\end{cor}

\section{Type III full groups}\label{sec:typeIII}

We begin this section by recasting the definitions from the introduction in the more general setting of non-singular equivalence relations.

Two Borel probability measures are \textbf{equivalent} if they have the same null sets.

A Borel bijection of the standard probability space $(X,\mu)$ is \textbf{non-singular} if the probability measure $T_*\mu$ defined by $T_*\mu(A)=\mu(T\inv(A))$ is equivalent to $\mu$. Similarly, a Borel action of a countable group $\Gamma$ on $(X,\mu)$ is called non-singular if all the elements of $\Gamma$ define non-singular  Borel bijections of $(X,\mu)$.

A Borel equivalence relation on $X$ is called \textbf{countable} if all its classes are countable. For instance, if $\Gamma$ is a countable group acting on $X$ in a Borel manner, define the equivalence relation $\mathcal R_{\Gamma}$ by $$(x,y)\in\mathcal R_\Gamma\iff x\in\Gamma\cdot y.$$  Then $\mathcal R_\Gamma$ is a countable Borel equivalence relation. In fact, by a result of Feldman and Moore \cite{MR0578656}, every countable Borel equivalence relation arises this way: for every countable Borel equivalence relation $\mathcal R$, there exists a countable group $\Gamma$ and a Borel action of $\Gamma$ on $X$ such that $\mathcal R=\mathcal R_{\Gamma}$. 

Define the \textbf{pre-full group} of a countable Borel equivalence relation $\mathcal R$ on $(X,\mu)$ to be the group of all Borel bijections $T:X\to X$ such that for all $x\in X$, $(x,T(x))\in\mathcal R$.

\begin{df}A countable Borel equivalence relation $\mathcal R$ on $(X,\mu)$ is \textbf{non-singular} if every element of its pre-full group is non-singular. 
\end{df}

It is well-known that a countable Borel equivalence relation $\mathcal R$ on $(X,\mu)$ is non-singular if and only if it comes from a non-singular action of a countable group (see e.g. \cite[Prop. 8.1]{MR2095154}).

The \textbf{full group} $[\mathcal R]$ of a non-singular equivalence relation $\mathcal R$ is by definition the quotient of its pre-full group by the normal subgroup consisting of all elements whose support have measure zero\footnote{Note that the fact that this is a normal subgroup is a consequence of the fact that $\mathcal R$ is non-singular.}. Equivalently, one can see the full group as a subgroup of $\Aut^*(X,\mu)$, where $\Aut^*(X,\mu)$ denotes the group of all non-singular Borel bijections of $(X,\mu)$ up to equality on a full measure set.

Let $\mathcal R$ and $\mathcal S$ be two non-singular equivalence relations. Then a map $T\in\Aut^*(X,\mu)$ is said to induce an \textbf{orbit equivalence} between $\mathcal R$ and $\mathcal S$ if it maps bijectively almost every $\mathcal R$-orbit to an $\mathcal S$-orbit, that is, if for almost all $x\in X$,
$$T([x]_{\mathcal R})=[T(x)]_{\mathcal S}.$$
It is easy to see that $T\in\Aut^*(X,\mu)$ induces an orbit equivalence between $\mathcal R$ and $\mathcal S$ if and only if $T[\mathcal R]T\inv=[\mathcal S]$. Let us now give the most basic invariants of orbit equivalence.

A non-singular equivalence relation $\mathcal R$  on $(X,\mu)$ is \textbf{ergodic} if every Borel set $A\subseteq X$ that is a reunion of $\mathcal R$-classes has measure $0$ or $1$. If $\mathcal R$ is an ergodic non-singular equivalence relation, we say that\begin{itemize}
\item $\mathcal R$ is \textbf{type II$_1$} if there exists a probability measure $\nu$ equivalent to $\mu$ which is preserved by $\mathcal R$, that is, such that every element $T$ of the pre-full group of $\mathcal R$ satisfies $T_*\nu=\nu$;
\item $\mathcal R$ is \textbf{type II$_\infty$} if there exists a $\sigma$-finite measure $\nu$ equivalent to $\mu$ which is preserved by $\mathcal R$;
\item $\mathcal R$ is \textbf{type III} otherwise.
\end{itemize}
These cases are mutually exclusive, and here, we will focus on type III ergodic equivalence relations. Let us give one more definition.

\begin{df}The \textbf{pseudo-full group} $[[\mathcal R]]$ of a non-singular equivalence relation $\mathcal R$ is the set of all injective Borel maps $\varphi$ from $X$ to $X$ such that for all $x\in\dom\varphi$, one has $(x,\varphi(x))\in\mathcal R$.
\end{df}

The following proposition is folklore; since we could not find a precise reference in the literature, we provide a proof. 

\begin{prop}\label{prop:transactiononmalg}Let $\mathcal R$ be a type III ergodic equivalence relation on $(X,\mu)$. Then for every Borel subsets $A,B\subseteq X$ of positive measure, there exists $\varphi\in[[\mathcal R]]$ such that $\dom\varphi=A$ and $\rng\varphi=B$.
\end{prop}

The proof follows \cite[Lem. 6]{MR0369658} closely, where the same statement is proved in the case when $\mathcal R$ is generated by a single type III  ergodic automorphism of $(X,\mu)$. We need the following lemma, which corresponds to \cite[Lem. 3]{MR0369658}.

\begin{lem}\label{lem:restriction type III}
Let $\mathcal R$ be a type III ergodic equivalence relation on $(X,\mu)$. Then for every Borel subset $A$ of $X$ with positive measure, the restriction of $\mathcal R$ to $A$ is type III.
\end{lem}
\begin{proof}
Let $\Gamma$ be a countable group such that $\mathcal R=\mathcal R_\Gamma$. Then by ergodicity, one can find a partition $(B_i)_{i\in\N}$ of $X\setminus A$ such that for every $i\in\N$, there exists $\gamma_i\in\Gamma$ such that $\gamma_i(B_i)\subseteq A$. Suppose that the restriction of $\mathcal R$ to $A$ is not type III: then it preserves a $\sigma$-finite measure $\nu$ equivalent to the restriction of $\mu$ to $A$. We then extend $\nu$ to $X$ via the $\gamma_{i}$'s: for every Borel subset $B$ of $X$, let 
$$\eta(B)=\nu(A\cap B)+\sum_{i\in\N} \nu(\gamma_i(B_i\cap B)).$$
The $\sigma$-finite measure $\eta$ we obtain is $\mathcal R$-invariant and equivalent to $\mu$, contradicting the fact that $\mathcal R$ was type III.
\end{proof}

\begin{proof}[Proof of Proposition \ref{prop:transactiononmalg}]
Let $\mathcal R$ be a type III ergodic equivalence relation on $(X,\mu)$. We introduce the following notation: for two Borel subsets $A,B\subseteq X$, write $A\prec_{\mathcal R} B$ if there exists $\varphi\in[[\mathcal R]]$ whose domain is equal to $A$ while its range is included in $B$.  By the Borel version of the Schröder-Bernstein theorem (see \cite[Thm. 15.7]{MR1321597}), we see that if $A\prec_{\mathcal R} B$ and $B\prec_{\mathcal R} A$, then there exists $\varphi\in[[\mathcal R]]$ whose domain is equal $A$ and whose range is actually equal to $B$. So by symmetry, the proof boils down to showing that for any $A,B\subseteq X$ of positive measure, one has $A\prec_{\mathcal R} B$. 

To this end, consider the restriction $\mathcal R_{\restriction B}$ of $\mathcal R$ to $B$: by the previous lemma, it is type III. Let $\Gamma$ be a countable group acting on $B$ such that $\mathcal R_{\restriction B}=\mathcal R_\Gamma$. Applying \cite[Thm. 1]{MR0240283} to the $\Gamma$-action on $B$, one finds a positive measure subset $W\subseteq B$ and an infinite subset $I\subseteq \Gamma$ such that for all $\gamma\neq\gamma'\in I$, the sets $\gamma(W)$ and $\gamma'(W)$ are disjoint.

Since $\mathcal R$ is ergodic, we can find a partition $(A_\gamma)_{\gamma\in I}$ of $A$ such that for every $\gamma\in I$ one has $A_\gamma\prec_{\mathcal R} W$. But then for all $\gamma\in I$, we have $A_\gamma\prec_{\mathcal R}\gamma(W)$, and since $(\gamma(W))_{\gamma\in I}$ is a partition of a subset of $B$, we deduce that $A\prec_{\mathcal R}B$.
\end{proof}

\section{Full groups as closed subgroups of \texorpdfstring{$\LL^0(X,\mu,\mathfrak S_\infty)$}{L0(X,mu,S infinity)}}

This section is devoted to the proof of Theorem \ref{thm:fullgroupag1}: we want to exhibit examples of full groups with ample generics. Let us first give some background on the group $\mathfrak S_\infty$ of all permutations of the set $\N$. It is equipped with the topology of pointwise convergence, which means that a sequence $(\sigma_n)$ of permutations converges to $\sigma$ if for every $k\in\N$, we have $\sigma_n(k)=\sigma(k)$ for $n$ large enough. An explicit basis of neighborhoods of the identity is provided by the sets $$\{\sigma\in\mathfrak S_\infty: \sigma_{\restriction\{0,...,N\}}=\mathrm{id}_{\{0,...,N\}}\}$$ where $N\in\N$. Our proof of Theorem \ref{thm:fullgroupag1} nicely parallels the proof of the fact that $\mathfrak S_\infty$ has ample generics, which was first shown by Hodges, Hodkinson, Lascar and Shelah (\cite{zbMATH00540578}). First, we need to define the topology on full groups.

Given a non-singular equivalence relation $\mathcal R$, endow its full group $[\mathcal R]$ with the uniform topology, induced by the metric $d_u$ defined by
$$d_u(S,T):=\mu(\{x\in X:S(x)\neq T(x)\}).$$
Then $[\mathcal R]$ is a Polish group.

An equivalence relation is called \textbf{aperiodic} if all its equivalence classes are infinite.
By an easy modification of the Lusin-Novikov theorem (see \cite[Ex. 18.15]{MR1321597}), if $\mathcal R$ is a aperiodic countable Borel equivalence relation, then there exists a countable family of Borel maps $(f_i)_{i\in\N}$ from $X$ to $X$ with \textit{disjoint} graphs, where $f_0$ is the identity map, such that 
$$\mathcal R=\bigsqcup_{i\in\N}\{(x,f_i(x)): x\in X\}.$$
Such a family $(f_i)_{i\in\N}$ is called a \textbf{decomposition} of $\mathcal R$. Note that if  $\mathcal R$ is moreover non-singular, then for all $i\in\N$, the pushforward measure $f_{i*}\mu$ is absolutely continuous with respect to $\mu$.

The following proposition clarifies the link between full groups and $\mathfrak S_\infty$. We will use this connection throughout the lemmas leading to Theorem \ref{thm:fullgroupag1}.

\begin{prop}\label{prop:fullgroupsubl0}Let $\mathcal R$ be a non-singular aperiodic equivalence relation on $(X,\mu)$. Let $(f_i)_{i\in\N}$ be a decomposition of $\mathcal R$. Then the function
$$\Phi: [\mathcal R]\to \LL^0(X,\mu,\mathfrak S_\infty)$$
which maps every $T\in[\mathcal R]$ to the measurable function $$\Phi(T):x\mapsto (i\in\N\mapsto \text{ the unique }j\in\N\text{ such that }T(f_i(x))=f_j(x))$$
is an embedding.
\end{prop}
\begin{proof}
It is easy to check that $\Phi$ is a well-defined group homomorphism. Next, we show that $\Phi$ is continuous. To this end, let $\epsilon>0$ and $N\in\N$ and consider the following basic open neighborhood of the identity in $\LL^0(X,\mu,\mathfrak S_\infty)$:
$$U_{\epsilon,N}=\left\{f\in\LL^0(X,\mu,\mathfrak S_\infty): \mu(\{x\in X: f(x)_{\restriction\{0,...,N\}}=\mathrm{id}_{\{0,...,N\}}\}) > 1 - \epsilon \right\}.$$
Let $(T_n)$ converge to the identity map of $X$ and for each $n \in \N$, let $A_n=X\setminus\supp T_n$. Then $\mu(A_n)$ tends to $1$ so for all $i\in\{0,...,N\}$, we have $\mu(f_i\inv(A_n))\to 0$. Thus, for $n$ large enough, we have 
$\mu(\bigcap_{i=0}^Nf_i\inv(A_n))>1-\epsilon$, so that $\Phi(T_n)\in U_{\epsilon,N}$. This means that $\Phi$ is continuous.

Moreover, since $\Phi\inv(U_{\epsilon,0})$ consists precisely of those $T\in[\mathcal R]$ whose support has measure less than $\epsilon$, we see that $\Phi$ is injective and is a homeomorphism onto its image, completing the proof. 
\end{proof}

\begin{rmq}
The above proof amounts to identifying $X\times \N$ and $\mathcal R$ via $(x,i)\mapsto (x,f_i(x))$. The right-action $\rho$ of $[\mathcal R]$ on $\mathcal R$ given by $\rho(T)(x,y)=(x,T(y))$ provides the embedding $\Phi:[\mathcal R]\to \LL^0(X,\mu,\mathfrak S_\infty)$ of the previous theorem. 
\end{rmq}

In what follows, we say that an $n$-tuple of permutations $(\sigma_1,...,\sigma_n)$ of a set $X$ and an $n$-tuple of permutations $(\tau_1,...,\tau_n)$ of a set $Y$ are \textbf{conjugate} if there exists a bijection $\sigma: X\to Y$ such that for all $i\in\{1,...,n\}$, we have $\sigma\sigma_i\sigma\inv=\tau_i$.

\begin{lem}\label{lem:gdeltasinfini} Let $E_n$ denote the set of $n$-tuples $(\sigma_1,...,\sigma_n)\in\mathfrak S_\infty^n$ such that the following two conditions are satisfied:
\begin{enumerate}[(1)]
\item \label{cond:allorbitsfinite}every $\la \sigma_1,...,\sigma_n\ra$-orbit is finite, and
\item \label{cond:manyorbitsconjtoalpha} for every transitive\footnote{By definition, an $n$-tuple $(\tau_1,...,\tau_n)$ of permutations of a finite set is \textbf{transitive} if the associated $\la \tau_1,...,\tau_n\ra$ action is transitive, i.e. has only one orbit.} $n$-tuple $(\tau_1,...,\tau_n)$ of permutations of a finite set, there exists infinitely many $\la \sigma_1,...,\sigma_n\ra$-orbits such that the restriction of $(\sigma_1,...,\sigma_n)$ to each of them is conjugate to $(\tau_1,...,\tau_n)$.
\end{enumerate}
Then $E_n$ is $G_\delta$ in $\mathfrak S_\infty^n$. 
\end{lem}
\begin{proof}
For a fixed $k\in\N$, consider an $n$-tuple $(\sigma_1,...,\sigma_n)\in\mathfrak S_\infty^n$ such that the $\la \sigma_1,...,\sigma_n\ra$-orbit of $k$ is finite. Let $F$ denote this orbit. If an $n$-tuple $(\tau_1,...,\tau_n)\in\mathfrak S_\infty^n$ has the same restriction to $F$ as $(\sigma_1,...\sigma_n)$, then the $\la \tau_1,...,\tau_n\ra$-orbit of $k$ is also equal to $F$, so in particular it is finite. Thus, the set of $n$-tuples $(\sigma_1,...,\sigma_n)\in\mathfrak S_\infty^n$ for which the orbit of $k$ is finite is an open subset of $\mathfrak S_\infty^n$.

Since all the $\la\sigma_1,...,\sigma_n\ra$-orbits are finite if and only if for all $k\in\N$, the $\la \sigma_1,...,\sigma_n\ra$-orbit of $k$ is finite, condition (\ref{cond:allorbitsfinite}) defines a $G_\delta$ subset of $\mathfrak S_\infty^n$. 

To complete the proof, we now need to see why condition (\ref{cond:manyorbitsconjtoalpha}) also defines a $G_\delta$ subset of $\mathfrak S_\infty^n$. 

 Up to conjugation, there are countably many $n$-tuples of permutations of a finite set. Moreover, since a countable intersection of $G_\delta$ sets is $G_\delta$, we only need to see that for a fixed transitive $n$-tuple $(\tau_1,...,\tau_n)$ of permutations of a finite set, the following condition on $(\sigma_1,...,\sigma_n)\in\mathfrak S_\infty^n$ defines a $G_\delta$ set: there are infinitely many $k\in\N$ such that the restriction of $(\sigma_1,...,\sigma_n)$ to the $\la\sigma_1,...,\sigma_n\ra$-orbit of $k$  is conjugate to $(\tau_1,...,\tau_n)$. 

For a fixed $k\in\N$, the same reasoning as for condition (\ref{cond:allorbitsfinite}) yields that the following set $C_k$ is open: the set of $(\sigma_1,...,\sigma_n)\in\mathfrak S_\infty^n$ such that the restriction of $(\sigma_1,...,\sigma_n)$ to the orbit of $k$ is conjugate to $(\tau_1,...,\tau_n)$. We deduce the set of $(\sigma_1,...,\sigma_n)$ such that there are infinitely many $k\in\N$ with $(\sigma_1,...,\sigma_n)\in C_k$ is $G_\delta$, which completes the proof. 
\end{proof}

The following lemma mirrors Lemma \ref{lem:gdeltasinfini}.

\begin{lem}\label{lem:fngdelta}
Let $\mathcal R$ be an aperiodic non-singular equivalence relation on $(X,\mu)$. Let $\mathcal E_n$ denote the set of $n$-tuples $(T_1,...,T_n)\in[\mathcal R]^n$ such that the following two conditions are satisfied:
\begin{enumerate}[(1)]
\item \label{cond:allorbitsfinite2}almost every $\la T_1,...,T_n\ra$-orbit is finite, and
\item \label{cond:manyorbitsconjtoalpha2} for every transitive $n$-tuple $(\tau_1,...,\tau_n)$ of permutations of a finite set and almost every $x\in X$, there are infinitely many $y\in[x]_\mathcal R$ such that the restriction of $(T_1,...,T_n)$ to the $\la T_1,...,T_n\ra$-orbit of $y$ is conjugate to $(\tau_1,...,\tau_n)$.
\end{enumerate}
Then $\mathcal E_n$ is $G_\delta$ in $[\mathcal R]^n$. 
\end{lem}
\begin{proof}
The set $E_n$ of Lemma \ref{lem:gdeltasinfini} is $G_\delta$ in $\mathfrak S_\infty^n$, so by Lemma \ref{lem:gdeltaallinG}, we have that $\LL^0(X,\mu;E_n)$ is a $G_\delta$ subset of $\LL^0(X,\mu;\mathfrak S_\infty^n)$. 
Through the natural identification between $\LL^0(X,\mu,\mathfrak S_\infty^n)$ and $\LL^0(X,\mu,\mathfrak S_\infty)^n$, we then see $\LL^0(X,\mu,E_n)$ as a $G_\delta$ subset of $\LL^0(X,\mu,\mathfrak S_\infty)^n$. 

But then, using the notations of Proposition \ref{prop:fullgroupsubl0}, we have that $$\mathcal E_n=(\Phi,...,\Phi)\inv(\LL^0(X,\mu;E_n)).$$ Since $\Phi$ is continuous and $E_n$ is $G_\delta$, we deduce that $\mathcal E_n$ is also $G_\delta$.
\end{proof}

\begin{rmq}We will see during the proof of Theorem \ref{thm:fullgroupag1} that when $\mathcal R$ is hyperfinite, $\mathcal E_n$ is a dense $G_\delta$ set for all $n\geq 1$ (we prove it only in the type III case, but it is true in general). This actually characterizes hyperfiniteness as a consequence of a result of Eisenmann and Glasner which was a great source of inspiration to us (see \cite[Thm. 1.6 (2)]{Eisenmann:2014oq}). 
\end{rmq}

The next lemma is a bit technical to state, and will only be used to show that $\mathcal E_n$ consists of a single conjugacy class in the type III case (Lemma \ref{lem:singleconjclass}). 

\begin{lem}\label{lem:conjugate in S_p} Let $p\in\N$, consider the product action of the cyclic group $\Z/p\Z$ on the space $(\Z/p\Z\times A, m\times \nu)$, where $\Z/p\Z$ acts on itself by translation and trivially on $A$, $m$ is the normalized counting measure on $\Z/p\Z$ and $(A,\nu)$ is a standard probability space.
Fix $(\tau_1,...,\tau_n)\in \mathfrak S_p^n$ and $(S_1,...,S_n), (T_1,...,T_n)\in[\mathcal R_{\Z/p\Z}]^n$ such that for every $x\in A\times \Z/p\Z$, both the $(S_1,...,S_n)$- and the $(T_1,...,T_n)$-actions on $[x]_{\mathcal R_{\Z/p\Z}}$ are conjugate to $(\tau_1,...,\tau_n)$. 

Then $(S_1,...,S_n)$ and $(T_1,...,T_n)$ are diagonally conjugate: there is $S\in[\mathcal R_{\Z/p\Z}]$ such that for all $i\in\{1,...,n\}$, 
$ST_iS\inv=S_i$.
\end{lem}

\begin{proof}
View $\mathfrak S_p$ as the group of permutations of $\Z/p\Z$. Then $\LL^0(A,\mu,\mathfrak S_p)$ acts on $\Z/p\Z\times A$ as follows: for every $f\in \LL^0(A,\mu;\mathfrak S_p)$ and $(g,x)\in \Z/p\Z\times A$,
$$f\cdot (g,x)=(f(x)g,x).$$
This yields a group homomorphism $\Psi: \LL^0(A,\mu;\mathfrak S_p)\to[\mathcal R_{\Z/p\Z}]$ which is easily seen to be an isomorphism. 

The group $\LL^0(A,\nu;\mathfrak S_p)$ acts on $\LL^0(A,\nu,\mathfrak S_p)^n=\LL^0(A,\nu,\mathfrak S_p^n)$ by diagonal conjugation. We have to show that $(\Psi\inv(S_i))_{i=1}^n$ and $(\Psi\inv(T_i))_{i=1}^n$  belong to the same orbit. By assumption, for all $x\in A$, both  $(\Psi\inv(S_i)(x))_{i=1}^n$  and $(\Psi\inv(T_i)(x))_{i=1}^n$ are conjugate to $(\tau_i)_{i=1}^n$. 

But by Equation (\ref{eqn:orbit}) from the proof of Theorem \ref{thm:les orbites comaigres passent aux L^0},  this implies that $(\Psi\inv(S_i))_{i=1}^n$ and $(\Psi\inv(T_i))_{i=1}^n$  both belong to the $\LL^0(A,\nu;\mathfrak S_p)$-orbit of the $n$-tuple of constant maps $(\overline{\tau_i})_{i=1}^n$, completing the proof.
\end{proof}

The following lemma is crucial and mirrors the fact that in $\mathfrak S_\infty^n$, the set $E_n$ defined in Lemma \ref{lem:gdeltasinfini} consists of a single diagonal conjugacy class. 
\begin{lem}\label{lem:singleconjclass}Let $\mathcal R$ be a type III ergodic equivalence relation. Then the set $\mathcal E_n$ defined in Lemma \ref{lem:fngdelta} consists of a single diagonal $[\mathcal R]$-conjugacy class.
\end{lem}
\begin{rmq}
The lemma is false if we do not assume $\mathcal R$ is type III. Indeed, if $\mathcal R$ preserves a $\sigma$-finite measure $\nu$, then the measure of the set of $x\in X$ with an orbit of size $n$ becomes an invariant of conjugacy.
\end{rmq}
\begin{proof}

Consider two $n$-tuples $(S_1,...,S_n)$ and $(T_1,...,T_n)$ in $\mathcal E_n$. We want to find $T\in[\mathcal R]$ such that for all $i\in\{1,...,n\}$, $TS_iT\inv=T_i$.

Up to conjugation, there are only countably many transitive $n$-tuples $(\tau_1,...,\tau_n)$ of permutations of finite sets. Let $(\tau_{1,k},...,\tau_{n,k})_{k\in\N}$ be a choice of representatives of transitive $n$-tuples of permutations of finite sets for the conjugacy equivalence relation. 
For every $k\in\N$, let 
\begin{align*}
A_k:=&\{x\in X:\text{the restriction of }(S_1,...,S_n)\text{ to the }\la S_1,...,S_n\ra\text{-orbit of } x\\
&\text{is conjugate to }(\tau_{1,k},...,\tau_{n,k})\}\\
B_k:=&\{x\in X:\text{the restriction of }(T_1,...,T_n)\text{ to the }\la T_1,...,T_n\ra\text{-orbit of } x\\
&\text{is conjugate to }(\tau_{1,k},...,\tau_{n,k})\}\end{align*}
Then each $A_k$ is clearly $\la S_1,...S_n\ra$-invariant while each $B_k$ is $\la T_1,...,T_n\ra$-invariant. Moreover, since $(S_1,...,S_n)\in\mathcal E_n$ and $(T_1,...,T_n)\in\mathcal E_n$, we see that both $(A_k)_{k\in\N}$ and $(B_k)_{k\in\N}$ are partitions\footnote{By convention, this also means that each $A_k$ (resp. $B_k$) has positive measure.} of $X$. 

For each $k\in\N$, we will build a map $\varphi_k: A_k\to B_k$ belonging to $[[\mathcal R]]$ which conjugates $(S_1,...,S_n)_{\restriction A_k}$ to $(T_1,...,T_n)_{\restriction B_k}$. Granting this, the element $T\in[\mathcal R]$ obtained by gluing together the $\varphi_k$'s will conjugate $(S_1,...,S_n)$ to $(T_1,...,T_n)$, which will complete the proof the lemma. 

For the construction, fix $k\in\N$ and let $p\in\N$ be the size of the orbit of $\la \tau_{1,k},...\tau_{n,k}\ra$. Then one may as well assume that $\tau_{1,k},...,\tau_{n,k}$ belong to the group $\mathfrak S_p$ of permutations of the set $\{0,...,p-1\}$. By assumption the $\la S_1,...,S_n\ra$-orbit of every $x$ in $A$ has cardinality $p$. Fix a Borel linear order $<$ on $X$ (as one may assume $X=[0,1]$, such an order exists). Then the set 
$$A=\{x\in A_k: x\text{ is the $<$-minimum of its}\la S_1,...,S_n\ra\text{-orbit}\}$$
is Borel and we define $U\in[\mathcal R]$ to be the map which is the identity outside of $A_k$ and sends every $x\in A_k$ to its successor for the cyclic order induced by $<$ on the $\la S_1,...,S_n\ra$-orbit of $x$.  By construction, we have that $(U^i(A))_{i=0}^{p-1}$ is a partition of $A_k$.

Similarly, we can find $B\subseteq B_k$ and $V\in[\mathcal R]$ such that $(V^i(B))_{i=0}^{p-1}$ is a partition of $B_k$. 

Since $\mathcal R$ is type III ergodic, by Proposition \ref{prop:transactiononmalg} there exists $\psi\in [[\mathcal R]]$ such that $\psi(A)=B$. Define $\psi_k: A_k\to B_k$ by: for all $i\in\{1,...,k\}$ and $x\in U^i(A)$, 
$$\psi_k(x)=V^i\psi U^{-i}(x).$$
Clearly, $\psi_k$ is an orbit equivalence between the equivalence relations induced by $(S_1,...,S_n)_{\restriction A_k}$ and $(T_1,...,T_n)_{\restriction B_k}$.

To conclude, note that the automorphism $U\in[\mathcal R]$ defined above yields an action of the cyclic group $\Z/p\Z$ on $A_k=\bigsqcup_{i=0}^{p-1} U^i(A)$ which is conjugate to the action defined in Lemma \ref{lem:conjugate in S_p}. We may thus apply this lemma to $(S_{i\restriction A_k})_{i=1}^k$ and $(\psi_k\inv T_{i\restriction B_k}\psi_k)_{i=1}^n$ for $(\tau_i)_{i=1}^n=(\tau_{i,k})_{i=1}^n$: we find $S\in[\mathcal R_{\Z/p\Z}]$ such that for all $i\in\{1,...,n\}$,
$$SS_{i\restriction A_k}S\inv=\psi_k\inv T_{i\restriction B_k}\psi_k.$$
Then, the map $\varphi_k=\psi_kS\in[[\mathcal R]]$ conjugates $S_{i\restriction A_k}$ to $T_{i\restriction B_k}$ for all $i\in\{1,...,n\}$, so the lemma is proven.
\end{proof}

We need one more lemma before we embark on the proof of Theorem \ref{thm:fullgroupag1}.

\begin{lem}\label{lem: Sinfini dans [R]}Let $\mathcal R$ be a type III ergodic equivalence relation. Then there exists a subrelation $\mathcal S\subseteq \mathcal R$ such that for every $(\sigma_1,...,\sigma_n)\in\mathfrak S_\infty^n$, there exists $(T_1,...,T_n)\in[\mathcal S]$ such that for almost every $x\in X$, the $(T_1,...,T_n)$-action on $[x]_\mathcal S$ is conjugate to the $(\sigma_1,...,\sigma_n)$-action on $\N$.
\end{lem}
\begin{proof}
For this proof, it is easier to think of $\mathfrak S_\infty$ as the group of permutations of $\Z$.
Let us write $X$ as a countable partition $(A_k)_{k\in\Z}$, where each $A_k$ has positive measure. By Proposition \ref{prop:transactiononmalg}, for each $k\in\Z$ we may and do fix $\varphi_k: A_k\to A_{k+1}$ belonging to $[[\mathcal R]]$. Gluing these $\varphi_k$'s together, we obtain $T\in[\mathcal R]$ and let $\mathcal S$ denote the equivalence relation induced by $T$. 

There is a natural homomorphism $\Psi:\mathfrak S_\infty\to [\mathcal S]$ given by: for all $\sigma\in\mathfrak S_\infty$,  all $k\in\Z$, and all $x\in A_k$, 
$$\Psi(\sigma)(x)=T^{\sigma(k)-k}(x).$$
Note that for all $x\in A_0$, we have $\Psi(\sigma)(T^k(x))=T^{\sigma(k)}(x)$. 
It is then easy to check that  for every $(\sigma_1,...,\sigma_n)\in\mathfrak S_\infty^n$, the elements of the full group of $\mathcal S$ defined by $T_i=\Psi(\sigma_i)$ satisfy the requirements of the lemma.
\end{proof}

A non-singular equivalence relation $\mathcal R$ on $(X,\mu)$ is \textbf{hyperfinite} if it can be written as an increasing union of finite Borel equivalence subrelations (i.e. with finite equivalence classes). By a result of Krieger \cite[Thm. 4.1]{MR0240279}, which is actually true in the pure Borel setting (see \cite{weiss1984measurable,slaman1988definable}), these arise exactly as the equivalence relations induced by non-singular $\Z$-actions. So Theorem \ref{thm:fullgroupag1} may be reformulated as follows.

\begin{thm}\label{thm:fullgrphyphasag}Let $\mathcal R$ be a hyperfinite type III ergodic equivalence relation. Then the full group of $\mathcal R$ has ample generics. 
\end{thm}
\begin{proof}
Fix $n\in\N$. By Lemma \ref{lem:singleconjclass}, the set $\mathcal E_n$ consists of a single $[\mathcal R]$-conjugacy class and it is $G_\delta$ by Lemma \ref{lem:fngdelta}. 

Therefore, we only have to show that $\mathcal E_n$ is dense. Since $\mathcal R$ is hyperfinite, it can be written as an increasing union of finite Borel equivalence subrelations $\mathcal R=\bigcup_{k\in\N} \mathcal R_k$. It follows that $\bigcup_{k\in\N} [\mathcal R_k]$ is dense in $[\mathcal R]$ (see \cite[Thm. 4.7]{MR2599891} whose proof readily adapts to the non-singular case). 

But then, this implies that the set $\mathcal F_n$ of $n$-tuples $(T_1,...,T_n)\in[\mathcal R]^n$ with only finite orbits is dense in $[\mathcal R]^n$. So it suffices to be able to approximate elements of $\mathcal F_n$ by elements of $\mathcal E_n$.

To this end, let $(T_1,...,T_n)\in\mathcal F_n$ and fix $\epsilon>0$. Since $(T_1,...,T_n)$ only has finite orbits, we can find a $(T_1,...,T_n)$-invariant set $A$ of measure less than $\epsilon$. 

Let us fix some $(\sigma_1,...,\sigma_n)\in E_n$ with the notation of Lemma \ref{lem:gdeltasinfini}. The restriction of $\mathcal R$ to $A$ is type III ergodic by Lemma \ref{lem:restriction type III}, so we can apply Lemma \ref{lem: Sinfini dans [R]} to it. We thus find an equivalence subrelation $\mathcal S\subseteq \mathcal R_{\restriction A}$ and $T'_1,...,T'_n\in[\mathcal S]$ such that the action of $(T'_1,...,T'_n)$ on almost every $\mathcal S$-class is conjugate to the action of $(\sigma_1,...,\sigma_n)$ on $\N$.
 
We then define $(\tilde T_1,...,\tilde T_n)\in[\mathcal R]^n$ by: for every $i\in\{1,...,n\}$ and $x\in X$,
$$\tilde T_i(x)=\left\{\begin{array}{ll}T'_i(x) & \text{if }x\in A, \\T_i(x) & \text{otherwise.}\end{array}\right.$$
  Because $\mu(A)<\epsilon$, each $\tilde T_i$ is $\epsilon$-close to $T_i$. Moreover by ergodicity almost every $\mathcal R$-class meets $A$, which implies by the construction of $(T'_1,...,T'_n)$ that the newly obtained $(\tilde T_1,...,\tilde T_n)$ belong to $\mathcal E_n$. Thus $\mathcal E_n$ is dense, which ends the proof.
\end{proof}

\begin{rmq}Using the framework presented in \cite[Sec. 2.1]{lm14nonerg}, one can remove the ergodicity assumption in Theorem \ref{thm:fullgrphyphasag}. 
\end{rmq}

\section{Further remarks}\label{sec:furthrem}

Using Proposition \ref{prop:transactiononmalg}, we see that Kitrell and Tsankov's result on automatic continuity for ergodic full groups of type II$_1$  (\cite[Thm. 3.1]{MR2599891}) readily adapts to the type III ergodic case. We do not know if all full groups of type III equivalence relation satisfy the automatic continuity property.

In the proof of Theorem \ref{thm:fullgrphyphasag} for $n=1$, we can replace hyperfiniteness by Rohlin's lemma (see \cite[Thm. 1]{chacon1965approximation}) and obtain that $[\mathcal R]$ has a comeager conjugacy class whenever $\mathcal R$ is a type III ergodic equivalence relation. We conjecture that however, as soon as $\mathcal R$ is not hyperfinite, $[\mathcal R]$ does not have ample generics\footnote{As was pointed out before, a good clue for this is the fact that the conjugacy class $\mathcal E_n$ ceases to be dense.}. 
 
Let us mention another result on full groups of type III which answers a question which was asked by Houdayer and Paulin to the second author. 

\begin{thm}Whenever $\mathcal R$ is a type III ergodic equivalence relation, its full group has topological rank\footnote{The topological rank of a topological group is the minimal number of elements needed to generate a dense subgroup.} 2. 
\end{thm}
\begin{proof}[Sketch of proof]
First, one can use the ideas in \cite{MR1116644} to show that $[\mathcal R]$ contains a type III ergodic element $T_0$ which is a non-singular odometer. Then, since by \cite[Lem. 5.2]{MR2311665} the group of dyadic permutations is dense in $[T_0]$, the exact same proof as for \cite[Thm. 4.1]{lm14nonerg} yields that one can find $U\in[T_0]$ of small support such that whenever $C\in[\mathcal R]$ has order 3, then the closed subgroup generated by $T$ and $UC$ contains $[T_0]$. 

Now, using Proposition \ref{prop:transactiononmalg}, one can show that the equivalence relation $\mathcal R$ is generated by $T_0$ and $\varphi\in[[\mathcal R]]$ whose range is disjoint from its domain and such that $\mu(\dom\varphi\cup\rng\varphi)<1$. Then, again using Proposition \ref{prop:transactiononmalg}, one can find $\psi\in[[\mathcal R_0]]$ with $\dom\psi=\rng\varphi$ and such that $\rng\psi$ is disjoint from $\dom\varphi\cup \rng\varphi$: in the terminology of \cite{gentopergo}, $\{\varphi,\psi\}$ is a pre-3-cycle. Let $C$ be the associated $3$-cycle. Then one can conclude that $\la T,UC\ra$ is a dense subgroup of $[\mathcal R]$ exactly as in \cite{gentopergo}, noting that Kittrell and Tsankov's theorem \cite[Thm. 4.7]{MR2599891} holds in the non-singular case. 
\end{proof}
\begin{rmq}Roughly speaking, the above result follows from the fact that any type III ergodic equivalence relation has cost 1. Using the same proof as for \cite[Prop. 5.1]{lm14nonerg}, one can show that for any type III ergodic equivalence relation $\mathcal R$, the set of pairs $(T,U)\in(\mathrm{APER}\cap[\mathcal R])\times [\mathcal R]$ which generate a dense subgroup of $[\mathcal R]$ is a dense $G_\delta$ set.
\end{rmq}


\bibliographystyle{alpha}
\bibliography{/Users/francoislemaitre/Dropbox/Maths/biblio}

\end{document}